\documentclass[11pt]{amsart}
\usepackage{amsmath, amssymb,verbatim,appendix}
\usepackage[mathscr]{eucal}
\usepackage{amscd}
\usepackage{amsthm}
\usepackage{stmaryrd}
\usepackage{enumerate}
\usepackage{comment}
\usepackage[latin1]{inputenc}
\usepackage{tikz}
\usetikzlibrary{shapes,arrows}
\usepackage{cite}

\usepackage[margin=1in]{geometry}

\AtBeginDocument{%
   \def\MR#1{}
}

\newtheorem{theorem}{Theorem}
\newtheorem{lemma}[theorem]{Lemma}
\newtheorem{corollary}[theorem]{Corollary}
\newtheorem{proposition}[theorem]{Proposition}
\newtheorem{remark}[theorem]{Remark}

\newtheorem{definition}[theorem]{Definition}

\newtheorem{claim}[theorem]{Claim}

\newcommand{\abar}{\bar{a}}

\title{
On unavoidable induced subgraphs in large prime graphs 
}

\author{M. Malliaris}
\address{Department of Mathematics, University of Chicago, 5734 S. University Avenue, Chicago, IL 60637}
\email{mem@math.uchicago.edu}

\author{C. Terry}
\address{Department of Mathematics, Statistics and Computer Science, University of Illinois at Chicago, 851 S. Morgan Street, Chicago, IL, 60607}
\email{cterry3@uic.edu}

\thanks{C.T. and M.M. thank MSRI for partially supporting their semester visits in Spring 2014 
via NSF grant 0932078 000, where conversations about [2] influenced the current project. 
M.M. was partially supported by a Sloan fellowship and NSF grant 1300634.
}

\date{\today}

\begin{document}


\maketitle

\begin{abstract}
Chudnovsky, Kim, Oum, and Seymour recently established that any prime graph contains one of a short list of induced prime subgraphs \cite{CKOS}. 
In the present paper we reprove their theorem using many of the same ideas, but with the key model-theoretic ingredient of first determining 
the so-called amount of stability of the graph. This approach changes the applicable Ramsey theorem, improves the bounds and offers a different 
structural perspective on the graphs in question. Complementing this, we give an infinitary proof which implies the finite result. 
\end{abstract}

\section{Introduction}

\setcounter{theorem}{0}
\numberwithin{theorem}{section}

Recently Chudnovsky, Kim, Oum, and Seymour established that any prime graph contains one of a short list of induced prime subgraphs \cite{CKOS}. A module of a graph $G=(V,E)$ is a set of vertices $X\subseteq V$ such that any vertex $v\in V\setminus X$ is either connected or non-connected to all vertices in $X$. Prime graphs are graphs which contain no non-trivial modules. The interest in prime graphs arises from questions around so-called modular decompositions of graphs, as well as the fact that the celebrated Erd\H{o}s-Hajnal conjecture reduces to the case where the omitted graph is prime.

In the present paper we re-prove the main theorem of \cite{CKOS} making use of model-theoretic ingredients, in a way that improves the bounds and offers a different structural perspective on the graphs in question. Our background aim is to exemplify the usefulness of model-theoretic ideas in proofs in finite combinatorics. This approach complements that of \cite{MS}, where certain indicators of complexity which had been identified by people working in combinatorics coincided with model theoretic dividing lines, so could be characterized by means of model theory.

The model-theoretic contribution of the present argument may be described as follows. The proof of \cite{CKOS} proceeds by means of several cases, sketched in section 2 below, and applies Ramsey's theorem as a main tool. In \cite{MS} it was shown that Ramsey's theorem works much better when the graph is so-called stable, a finitization of an important structural property identified by model theory (for history, see the introduction to \cite{MS} or the original source \cite{classification}). Our approach in the present paper, then, is essentially to reconfigure the proof of \cite{CKOS} so that the procedure for extracting the given configurations is different depending on the degree of stability of the graph, and can take advantage of this additional structural information.  

We believe this approach raises some interesting questions about model theory's potential contribution to calibrating arguments about finite objects. We have not tried to construct examples showing the bound we obtain is optimal, in part because we believe that a further development of what might be called `model-theoretic Ramsey theory' in the spirit of \cite{MS} may, in general, allow for even finer calibrations in the finite setting. At the same time, it is important to add that model theory works here to amplify the combinatorial analysis rather than to replace it. Already in the present argument, the contribution of combinatorics in e.g. identifying definitions such as `module' (which is much stronger than, if in some sense analogous to, the model-theoretic notion of an indiscernible sequence) and in isolating the original collection of induced configurations appears essential.  It is the interaction of these ideas and perspectives which to us seems most interesting.  

Complementing this approach, the paper concludes with the proof of an infinite analogue of Theorem \ref{mainthm} which implies the finite version, but without explicit bounds.
 

\tableofcontents

\section{Definitions and notation}

In this section we state relevant definitions and notation, most of which, but not all, is from \cite{CKOS}.  Given a set $X$, let ${X\choose 2}=\{Y\subseteq X: |Y|=2\}$.  A graph is a pair $(V,E)$ where $V$ is a set of vertices and $E\subseteq {V\choose 2}$ is a set of edges. Unless otherwise stated, all of the following definitions and notation apply to both infinite and finite graphs.  Given a graph $G$, we write $xy$ as shorthand for the edge $\{x,y\}$.  We will often write $V(G)=V$ and $E(G)=E$.  A set of vertices $X$ inside a graph is called a \emph{module} if every vertex outside of $X$ is adjacent to every vertex in $X$ or non-adjacent to every vertex in $X$.  A module $X$ of a graph $G$ is called \emph{trivial} if $|X|=1$ or $X=V(G)$.  A graph $G$ is called \emph{prime} if it has no non-trivial modules.  We say a set of vertices $X$ is \emph{independent} if every pair of vertices is $X$ is non-adjacent, and we say $X$ is \emph{complete} if every pair of vertices in $X$ is adjacent.  We say a vertex $v$ is \emph{mixed} on a subset $X\subseteq V$ if there are $x,y\in X$ such that $vx\in E$ and $vy\notin E$.  Given a graph $G=(V,E)$, the \emph{compliment of $G$}, denoted $\overline{G}$, is the graph with vertex set $V$ and edge set ${V\choose 2}\setminus E$.  Given two graphs $G$ and $H$, we will say $G$ ``contains a copy of $H$" to mean there is an induced subgraph of $G$ which is isomorphic to $H$.  

We now introduce important structural configurations which will appear throughout the paper.  Fix an integer $n\geq 1$.
\begin{enumerate}[$\bullet$] 
\item A \emph{half-graph} of height $n$ is a graph with $2n$ vertices $a_1,\ldots, a_n, b_1,\ldots, b_n$ such that $a_i$ is adjacent to $b_j$ if and only if $i\leq j$. 
\item The \emph{bipartite half-graph} of height $n$, $H_n$, is a graph with $2n$ vertices $a_1,\ldots, a_n, b_1,\ldots, b_n$ such that $a_i$ is adjacent to $b_j$ if and only if $i\leq j$ and such that $\{a_1,\ldots, a_n\}$ and $\{b_1,\ldots, b_n\}$ are independent sets.
\item The \emph{half split graph} of height $n$, $H'_n$, is a graph with $2n$ vertices $a_1,\ldots, a_n, b_1,\ldots, b_n$ such that $a_i$ is adjacent to $b_j$ if and only if $i\leq j$ and such that $\{a_1,\ldots, a_n\}$ is an independent set and $\{b_1,\ldots, b_n\}$ is a complete set (a graph is a \emph{split graph} if its vertices can be partitioned into a complete set and an independent set).  
\item Let $H_{n,I}'$ be the graph obtained from $H'_n$ by adding a new vertex adjacent to $a_1,\ldots, a_n$ (and no others).  Let $H_n^*$ be the graph obtained from $H'_n$ by adding a new vertex adjacent to $a_1$ (and no others).
\item The \emph{thin spider} with $n$ legs is a graph with $2n$ vertices $a_1,\ldots, a_n, b_1,\ldots, b_n$ such that $\{a_1,\ldots, a_n\}$ is an independent set, $\{b_1,\ldots, b_n\}$ is a complete set, and $a_i$ is adjacent to $b_j$ if and only if $i=j$.  The \emph{thick spider} with $n$ legs is the compliment of the thin spider with $n$ legs.  In particular, it is a graph with $2n$ vertices $a_1,\ldots, a_n, b_1,\ldots, b_n$ such that $\{a_1,\ldots, a_n\}$ is an independent set, $\{b_1,\ldots, b_n\}$ is a complete set, and $a_i$ is adjacent to $b_j$ if and only if $i\neq j$.   A \emph{spider} is a thin spider or a thick spider.
\item A sequence of distinct vertices $v_0,\ldots, v_m$ in a graph $G$ is called a \emph{chain} from a set $I\subseteq V(G)$ to $v_m$ if $m\geq 2$ is an integer, $v_0, v_1\in I$, $v_2,\ldots, v_m\notin I$, and for all $i>0$, $v_{i-1}$ is either the unique neighbor or the unique non-neighbor of $v_i$ in $\{v_0,\ldots, v_{i-1}\}$.  The \emph{length} of a a chain $v_0,\ldots, v_m$ is $m$.
\end{enumerate}
Given an integer $m\geq 1$, $K_m$ denotes the complete graph on $m$.  Given integers $m,n$, $K_{m,n}$ denotes the complete bipartite graph with parts of sizes $m$ and $n$, that is, the graph with $m+n$ vertices $\{a_1,\ldots, a_m, b_1,\ldots, b_n\}$ such that $\{a_1,\ldots, a_m\}$ and $\{b_1,\ldots, b_n\}$ are independent and $a_i$ is adjacent to $b_j$ for each $1\leq i\leq m$ and $1\leq j\leq n$.  Given a graph $G=(V,E)$, the \emph{line graph} of $G$ is the graph $G'$ which has vertex set $V(G')=E(G)$ and edge set consisting of pairs of elements $e_1\neq e_2\in E(G)$ such that $e_1\cap e_2\neq \emptyset$.  Given an integer $m$, a \emph{path} of length $m$ is a set $v_0,\ldots, v_m$ vertices such that $v_i$ is ajacent to $v_j$ if and only if $j=i+1$ or $i=j+1$.  The \emph{$m$-subdivision} of a graph $G$ is the graph obtained from $G$ by replacing every edge in $G$ with an induced path of length $m+1$.  A perfect matching of height $n$ is the disjoint union $n$ edges, that is, a graph with $2n$ vertices $\{a_1,\ldots, a_n, b_1,\ldots, b_n\}$ such that $\{a_1,\ldots, a_n\}$ and $\{b_1,\ldots, b_n\}$ are independent and $a_i$ is adjacent to $b_j$ if and only if $i=j$.

Note that in all of these definitions except that of a chain and of an $m$-subdivision, it makes sense to replace $m$ and $n$ by any cardinals $\lambda$ and $\mu$.  In section 6, we will wish to discuss versions of some of these configurations where $m$ or $n$ is replaced by an infinite cardinal.  In those cases, we will use the same notation as laid out in this section.

\section{Outline of proof of main theorem from \cite{CKOS}}\label{oldproofoutline}

In this section we give an outline of the proof of Theorem \ref{mainthm} presented in \cite{CKOS}.  We do this to allow for comparison to the proofs we present in sections \ref{finitary} and \ref{infinitary}.  Our outline consists of the statements of the propositions from \cite{CKOS} which form the main steps in their proof,  then a flow chart illustrating the structure of the proof.  We think this outline is sufficient for understanding the global structure of the proof.  For more details we direct the reader to the original paper \cite{CKOS}.  Throughout $R(n_1,\ldots, n_k)$ denotes the smallest integer $m$ such for that any coloring of the edges of $K_m$ with $k$, there is complete graph on $n_i$ vertices in color $i$ for some $1\leq i\leq k$.

\begin{theorem}[Theorem 1.2 of \cite{CKOS}]\label{mainthm}
For every integer $n\geq 3$ there is $N$ such that every prime graph with at least $N$ vertices contains one of the following graphs or their compliments as an induced subgraph.
\begin{enumerate}[(1)]
\item The $1$-subdivision of $K_{1,n}$ (denoted by $K_{1,n}^{(1)}$). 
\item The line graph of $K_{2,n}$.
\item The thin spider with $n$ legs.
\item The bipartite half-graph of height $n$.
\item The graph $H'_{n,I}$.
\item the graph $H_n^*$.
\item A prime graph induced by a chain of length $n$.
\end{enumerate}
\end{theorem}
\noindent We will use the following fact from \cite{CKOS}.

\begin{proposition}[Corollary 2.3 from \cite{CKOS}]\label{cor2.3}
Let $t>3$.  Every chain of length $t$ contains a chain of length $t-1$ inducing a prime subgraph.
\end{proposition}

\noindent The following are the propositions which form the main steps of the proof of Theorem \ref{mainthm} in \cite{CKOS}.

\begin{proposition}[Proposition 3.1 from \cite{CKOS}]\label{3.1}
For all integers $n, n_1, n_2>0$, there is $N=f(n,n_1,n_2)$ such that every prime graph with an $N$-vertex independent set contains an induced subgraph isomorphic to
\begin{enumerate}[(1)]
\item a spider with $n$ legs,
\item $\overline{L(K_{2,n})}$,
\item the bipartite half-graph of height $n$,
\item the disjoint union of $n_1$ copies of $K_2$, denoted $n_1K_2$ (i.e. an induced matching of size $n_1$), or
\item the half split graph of height $n_2$.
\end{enumerate}
Specifically, $f(n,n_1,n_2)=2^{M+1}$ where $M=R(n_1+n,2n-1,n+n_2, n+n_2-1)$.
\end{proposition}

\begin{proposition}[Proposition 4.1 from \cite{CKOS}]\label{4.1}
Let $t\geq 2$ and $n,n'$ be positive integers.  Let $h(n,n',2)=n$ and 
$$
h(n,n',i)=(n-1)R(n,n,n,n,n,n,n,n',n',h(n,n',i-1))+1
$$
for an integer $i>2$.  Let $v$ be a vertex of a graph $G$ and let $M$ be an induced matching of $G$  consisting of $h(n,n',t)$ edges not incident with $v$.  If for each edge $e=xy$ in $M$, there is a chain of length at most $t$ from $\{x,y\}$ to $v$, then $G$ has an induced subgraph isomorphic to one of the following:
\begin{enumerate}[(1)]
\item $K_{1,n}^{(1)}$,
\item the bipartite half-graph of height $n$,
\item $\overline{L(K_{2,n})}$,
\item a spider with $n$ legs, or
\item the half split graph of height $n'$.
\end{enumerate}
\end{proposition}

\begin{proposition}[Proposition 5.1 of \cite{CKOS}]\label{5.1}
For every positive integer $n$, there exists 
$$
N=g(n)=4^{n-2}(n+1)+2(n-2)+1
$$
such that every prime graph having a half split graph of height at least $N$ as an induced subgraph contains a chain of length $n+1$ or an induced subgraph isomorphic to one of $H_{n,I}'$, $H_n^*$, $\overline{H^*_n}$.
\end{proposition}

In the flow chart below, the bold boxes denote steps which involve Ramsey's theorem.  A box with no descendants indicates that the conclusion of the theorem is satisfied in that case.  In this chart, the functions $f$, $h$, and $g$ are from Propositions \ref{3.1}, \ref{4.1}, and \ref{5.1} respectively.

\newpage

\tikzstyle{decision} = [diamond, draw, fill=white!20, 
    text width=4.5em, text badly centered, node distance=3cm, inner sep=0pt]
\tikzstyle{block} = [rectangle, draw, fill=blue!20, 
    text width=5em, text centered, rounded corners, minimum height=4em]
\tikzstyle{line} = [draw, -latex']
\tikzstyle{cloud} = [draw, ellipse,fill=red!20, node distance=3cm,
    minimum height=2em]
\tikzstyle{box} = [rectangle, draw, fill=white!20, 
    text width=12em, text centered, minimum height=1em]
    \tikzstyle{thickbox} = [rectangle, draw, ultra thick, fill=white!20, 
    text width=12em, text centered, minimum height=1em]
    \tikzstyle{widebox} = [rectangle, draw, fill=white!20, 
    text width=20em, text centered, minimum height=1em]
     \tikzstyle{widethickbox} = [rectangle, draw, ultra thick, fill=white!20, 
    text width=20em, text centered, minimum height=1em]
      \tikzstyle{widishbox} = [rectangle, draw, fill=white!20, 
    text width=13em, text centered, minimum height=1em]
     \tikzstyle{widishthickbox} = [rectangle, draw, ultra thick, fill=white!20, 
    text width=13em, text centered, minimum height=1em]
    \tikzstyle{narrowbox} = [rectangle, draw, fill=white!20, 
    text width=5em, text centered, minimum height=1em]
      \tikzstyle{asterisk} = [circle, fill=white!20, 
    text width=.05em, text centered, minimum height=.05em]
    
\begin{tikzpicture}[node distance = 2cm, auto]
        \node[narrowbox](init){Start};
    \node [box, below right of=init,node distance=4cm] (yes chain) {There is a chain of length $n+1$.};
     \node [box, below right of=yes chain,node distance=3cm] (cor) {By Proposition \ref{cor2.3}, there is a chain of length $n$ inducing a prime subgraph.};
    \node [box, below left of=init,node distance=4cm] (no chain) {There is no chain of length $n+1$.};
    \node [widethickbox, below of=no chain, node distance=3cm] (ramsey) {Set $m=f(n,h(n,g(n),n),g(n))$, $N=R(m,m)$ and assume $G$ is a prime graph of size $N$.  By Ramsey's theorem, we may assume there is an independent set of size $m$ (else work with the dual).};
    \node [box, below right of=ramsey,node distance=4cm] (no half split) {There is no half split graph of height $g(n)$.};
    \node [box, below left of=ramsey,node distance=4cm] (yes half split) {There is a half split graph of height $g(n)$.};
    \node [box, below of=yes half split] (five point one) {Apply Proposition \ref{5.1}.};
    \node [widishthickbox, below of=no half split] (three point one) {Apply Proposition \ref{3.1} with $n=n$, $n_1=h(n,g(n),n)$ and $n_2=g(n)$.};
     \node [widebox, below right of=three point one,node distance=5cm] (four) {Outcome (4) of Proposition \ref{3.1}.  $G$ has an induced matching with $h(n,g(n),n)$ edges.  Since $G$ is prime, for every pair of points $\{x,y\}$ and every vertex $v$, there is a chain from $\{x,y\}$ to $v$.  Since $G$ has no chains of length $n+1$, all such chains have length at most $n$.  Therefore $G$ satisfies the hypotheses of Proposition \ref{4.1} with $n=n$, $n'=g(n)$, and $t=n$. };
      \node [box, below left of=three point one,node distance=5cm] (one two or three) {Outcome (1), (2), or (3) of Proposition \ref{3.1}.};
       \node [thickbox, below of=four, node distance=4cm] (four point one) {Apply Proposition \ref{4.1}.}; 
 
    \path [line] (init) -- (no chain);
    \path [line] (yes chain) -- (cor);
    \path [line] (init) -- (yes chain);
    \path [line] (no chain) -- (ramsey);
    \path [line] (ramsey) -- (no half split);
    \path [line] (ramsey) -- (yes half split);
    \path [line] (no half split) -- (three point one);
    \path [line] (yes half split) -- (five point one);
    \path [line] (three point one) -- (four);
     \path [line] (three point one) -- (one two or three);
      \path [line] (four) -- (four point one);
\end{tikzpicture}

For the rest of the paper, given $n\geq 2$, let $N_{\ref{mainthm}}=N_{\ref{mainthm}}(n)$ be the bound obtained for Theorem \ref{mainthm} in \cite{CKOS}, that is, $N_{\ref{mainthm}}(n)=R(m,m)$ where $m=f(n,h(n,g(n),n),g(n))$.  

\begin{remark}\label{reduction}
Note this proof shows the following: a prime graph $G$ with an independent set of size $m$ and no chain of length $n+1$ satisfies the conclusion of the theorem.
\end{remark}

\section{Tree Lemma}

\setcounter{theorem}{0}
\numberwithin{theorem}{section}

In this section we prove a key lemma, Theorem \ref{3.5}, which allows us to improve the bounds in Theorem \ref{mainthm}. 
This lemma is \cite{MS} Theorem 3.5 tailored to the specific setting of graphs. 
\cite{MS} Theorem 3.5 handles arbitrary finite sets of formulas, and uses model-theoretic tools such as types and $R$-rank. 
The bounds there are computed in terms of several associated constants,  
including the VC-dimension which was used to bound the branching of the trees. 
For the purposes of the present argument, we give here a streamlined proof for the special case of graphs 
written with graph theorists in mind. Corollary \ref{treecormain} gives the bound in this case. 

We now state relevant versions of definitions and lemmas from \cite{MS}.

Recall that a \emph{tree} is a partial order $(P, \trianglelefteq)$ such that for each $p\in P$, the set $\{q\in P: p\triangleleft q\}$ is a well-order under $\trianglelefteq$.  Given an integer $n\geq 2$, define
\begin{align*}
2^{<n}=\bigcup_{i=0}^{n-1} \{0,1\}^i,
\end{align*}
where $\{0,1\}^0=\langle$ $\rangle$ is the \emph{empty string}, and for $i>0$, $\{0,1\}^i$ is the usual cartesian product.  This set has a natural tree structure given by $\eta \trianglelefteq \eta'$ if and only if $\eta=\langle$ $\rangle$ or $\eta$ is an initial segment of $\eta'$.  We will write $\eta \triangleleft \eta'$ to denote that $\eta \trianglelefteq \eta'$ and $\eta \neq \eta'$.  Given $\eta \in \{0,1\}^i$, let $|\eta|=i$ denote \emph{length} of $\eta$ (the length of the empty string $\langle$ $\rangle$ is $0$).  A main idea in the proof of Theorem \ref{3.5} is to take a graph $G=(V,E)$, and arrange $G$ into a tree by indexing its vertex set with elements of $2^{<n}$.  Suppose $G=(V,E)$ is a graph, and we have an indexing $V=\{a_{\eta}: \eta \in X\}$ of the vertices of $G$ by some $X\subseteq 2^{<n}$.  Given $\eta \in X$, we will say the \emph{height} of $a_{\eta}$, denoted $ht(a_{\eta})$ is $|\eta|$.  A \emph{branch} is a set of the form $\{a_{\eta}: \eta \in Y\}$ where $Y$ is a maximal collection of comparable elements in $X$.  The \emph{length} of a branch is its cardinality.  Given $\eta, \eta' \in 2^{<n}$ and elements $a_{\eta}, a_{\eta'}$ indexed by $\eta$ and $\eta'$, we say $a_{\eta}$ and $a_{\eta'}$ \emph{lie along the same branch} if $\eta \trianglelefteq \eta'$ or $\eta'\trianglelefteq \eta$.  If $\eta\triangleleft \eta'$, we say $a_{\eta}$ \emph{precedes} $a_{\eta'}$.  Given $\eta=\langle \eta_1,\ldots, \eta_i\rangle \in \{0,1\}^i$, set $\eta \wedge 0= \langle \eta_1,\ldots, \eta_i,0\rangle $ and $\eta \wedge 1= \langle \eta_1,\ldots, \eta_i,1\rangle $.  If $x=a_{\eta\wedge 0}$ or $x=a_{\eta \wedge 1}$, then we say $a_{\eta}$ is the \emph{immediate predecessor} of $x$ and write $pred(x)=a_{\eta}$.  We will also write $a_{\eta}\wedge i$ to mean $a_{\eta\wedge i}$.  Given $j\in \{0,1\}$ and $i\geq 1$, let $j^i$ denote the element of $\{0,1\}^i$ which has every coordinate equal to $j$.  

\begin{definition}
Given a graph $G=(V,E)$ on $n$ vertices and $A\subseteq 2^{<n}$, we say that an indexing $V=\{a_{\eta}: \eta \in A\}$ of $V$ by the elements of $A$ is a \emph{type tree}, if for each $\eta \in A$ the following holds.  
\begin{itemize}
\item If $\eta \wedge 0 \in A$, then $a_{\eta \wedge 0}$ is non-adjacent to $a_{\eta}$.  If $\eta\wedge 1\in A$, then  $a_{\eta \wedge 1}$ is adjacent to $a_{\eta}$.
\item If $\eta \wedge 0$ and $\eta \wedge 1$ are both in $A$, then for all $\eta'\triangleleft \eta$, $a_{\eta \wedge 1}$ is adjacent to $a_{\eta'}$ if and only if $a_{\eta \wedge 0}$ is adjacent to $a_{\eta'}$.
\end{itemize}
\end{definition}

This notion of type tree is a special case of the model theoretic notion of a type tree.  We believe for the purposes of this paper it is better to deal only with this special version for graphs.  For the general definition, see \cite{classification}.   

\begin{lemma}\label{index}
Every finite graph $G=(V,E)$ can be arranged into a type tree.
\end{lemma}
\begin{proof}
Suppose $|V|=n$.  We arrange the vertices of $G$ into a type tree indexed by a subset of $2^{<n}$.  
\begin{itemize}
\item Stage 1: Choose any element of $G$ to be $a_{\langle \rangle}$, and set $A_0=\{a_{\langle \rangle}\}$.  Set $X_1=N(a_{\langle \rangle})$ and $X_0=V\setminus (\{a_{\langle \rangle}\}\cup N(a_{\langle \rangle}))$.  Note $X_1, X_0$ partition $V\setminus A_0$.

\item Stage $m+1$.  Suppose we've defined elements in the tree up to height $m\geq 0$ and for each $0\leq i\leq m$, $A_i$ is the set vertices of height $i$.  Suppose further that we have a collection of sets of vertices $\{X_{\eta \wedge i}: \eta \in A_m, i\in \{0,1\}\}$ which partition $V\setminus \bigcup_{i=1}^{m}A_i$ and such that for each $\eta \in A_m$, $X_{\eta\wedge 1}\subseteq N(a_{\eta})$ and $X_{\eta\wedge 0}\subseteq V\setminus (N(a_{\eta})\cup \{a_{\eta}\})$.   Then for each $\eta \in A_m$ and $i\in \{0,1\}$, if $X_{\eta \wedge i}\neq \emptyset$, choose $a_{\eta\wedge i}$ to be any element of $X_{\eta \wedge i}$.  Define $A_{m+1}$ to be the set of these $a_{\eta\wedge i}$.  Now for each $a_{\nu}\in A_{m+1}$ and $i\in \{0,1\}$, set
\begin{align*}
X_{\nu \wedge 1}&=N(a_{\nu})\cap X_{\nu}\text{ and }\\
X_{\nu \wedge 0}&=(V\setminus (N(a_{\nu})\cup \{a_{\nu}\}))\cap X_{\nu}.
\end{align*}
By assumption, $\{X_{\nu}: \nu \in A_{m+1}\}$ is a partition of $V\setminus \bigcup_{i=1}^{m}A_i$, and by construction, for each $\nu \in A_{m+1}$, $\{X_{\nu\wedge 1}, X_{\nu\wedge 0}\}$ is a partition of $X_{\nu}\setminus A_{m+1}$.  Therefore, $\{X_{\nu\wedge i}: \nu \in A_{m+1}, i\in \{0,1\}\}$ is a partition of $V\setminus \bigcup_{i=1}^{m+1}A_i$.
\end{itemize}
All elements of $V$ will be chosen after at most $n$ steps.  So we obtain an indexing of $V$ by a subset of $2^{<n}$ which is a type tree by construction.
\end{proof}

\begin{definition}\label{rankandheight} Suppose $G=(V,E)$ is a finite graph.
\begin{enumerate}
\item The \emph{tree rank} of $G$, denoted $t(G)$, is the largest integer $t$ such that there is a subset $V'\subseteq V$ and an indexing $V'=\{a_{\eta}:\eta \in 2^{<t}\}$ which is a type tree (i.e. $V'$ is a full binary type tree of height $n$).
\item The \emph{tree height} of $G$, denoted $h(G)$, is the smallest integer $h$ such that every indexing of $V$ which is a type tree has a branch of length $h$.
\end{enumerate}
\end{definition}

\begin{lemma}\label{indset}
Suppose $t,h$ are integers, and $G=(V,E)$ is a finite graph with tree rank $t$ and tree height $h$.  Then $G$ contains a complete or independent set of size $\max\{t, h/2\}$.
\end{lemma}
\begin{proof}
By definition of tree rank, there is $V'\subseteq V$ and an indexing $V'=\{a_{\eta}: \eta \in 2^{<t}\}$ which is a type tree.  Then by definition of a standard type tree, $I_1=\{a_{<>}, a_0, \ldots, a_{0^{t-1}}\}$ is an independent set of size $t$.  On the other hand, by definition of tree height and Lemma \ref{index}, there is an indexing $V=\{a_{\eta}: \eta \in B\}$ of $V$ by a subset $B\subseteq 2^{<n}$ which is a standard type tree and which contains a branch $J$ with length $h$.  Let $a_{\tau}$ be the last element of $J$ and note $h=ht(a_{\tau})$.  If  $|N(a_{\tau})\cap J|\geq \frac{|J|}{2}$, set $I_2 = N(a_{\tau})\cap J$.  Otherwise set $I_2 = (V\setminus N(a_{\tau}))\cap J$.  In either case, $|I_2|\geq |J|/2=h/2$.  We now show that $I_2$ is complete or independent.  Suppose $x$ and $y$ are elements of $I_2$.  By definition of $I_2$, $a_{\tau}$ is adjacent to $x$ if and only if $a_{\tau}$ is adjacent to $y$.  Note $x$ and $y$ lie along the same branch, so without loss of generality we may assume $x$ precedes $y$. By construction, $a_{\tau}$ is adjacent to $x$ if and only if $y$ is adjacent to $x$.  So if $I_2=N(a_{\tau})\cap J$, $I_2$ must be a complete set, and if $I_2=(V\setminus N(a_{\tau}))\cap J$, $I_2$ must be an independent set.  We've now shown $G$ contains a complete or independent set of size $\max\{|I_1|,|I_2|\}\geq \max\{t, h/2\}$.
\end{proof}

\begin{definition} Suppose $G=(V,E)$ is a graph, $A\subseteq 2^{<n}$, and $V=\{a_{\eta}: \eta \in A\}$ is a type tree.
\begin{enumerate}
\item Given an element $a_{\eta}\in V$, we say there is a \emph{full binary tree of height $k$ below $a_{\eta}$} if the following holds.  There is a set $V'\subseteq \{a_{\sigma}: a_{\eta}\subseteq a_{\sigma}\}$ and a bijection $f: V'\rightarrow 2^{<k}$ with the property that $a_{\sigma}$ precedes $a_{\sigma'}$ in $V'$ if and only if $f(a_{\sigma})\triangleleft f(a_{\sigma'})$ in $2^{<k}$.  
\item The \emph{tree rank} of an element $a_{\eta}\in V$, denoted $t(a_{\eta})$, is the largest $k$ such that there is a full binary tree of height $k$ below $a_{\eta}$.
\end{enumerate}
\end{definition}

\begin{theorem}\label{3.5}
Suppose $n\geq2$ is an integer and $G=(V,E)$ is a graph of size $n$.  Then 
$$
h(G)\geq \frac{(n/t(G))^{\frac{1}{t(G)+1}}}{2}.
$$
\end{theorem}
\begin{proof} Suppose $A\subseteq 2^{<n}$ and $V=\{a_{\eta}: \eta \in A\}$ of $V$ is a type tree.  Let $h$ be the length of the longest branch in this tree, and let $t=\max\{t(a_{\eta}): \eta \in A\}$.  Note $t\leq t(G)$.  Given a fixed $\ell$ and $s$, set
\begin{align*}
Z_{\ell}^s&=\{a_{\eta}\in V: t(a_{\eta})=s, ht(a_{\eta})=\ell\}\\
X_{\ell}^s&=\{a_{\eta}\in Z_{\ell}^s: t(p(a_{\eta}))=s\},\text{ and }\\
Y_{\ell}^s&=\{a_{\eta}\in Z_{\ell}^s: t(p(a_{\eta}))=s+1\}.
\end{align*}
Let $N_{\ell}^s=|Z_{\ell}^s|$, $x_{\ell}^s=|X_{\ell}^s|$ and $y_{\ell}^s=|Y_{\ell}^s|$.  Then note that that for each $s$ and $\ell$, $N^s_{\ell}=x_{\ell}^s+y_{\ell}^s$, and $n=\sum_{\ell=0}^h \sum_{s=0}^{t} N_{\ell}^s$.  We claim the following facts hold.
\begin{enumerate}[(i)]
\item For all $s\leq t$ and $\ell$, $x_{\ell+1}^s\leq N_{\ell}^s$.
\item For all $s<t$ and all $\ell$, $y^s_{\ell+1}\leq 2N_{\ell}^{s+1}$.
\item For all $s<t$ and all $\ell$, $N_{\ell+1}^s\leq N_{\ell}^s+2N_{\ell}^{s+1}$.
\item For all $1\leq s\leq t$, $N_0^{t-s}=0$.
\item For all $\ell$, $N^t_{\ell}\leq 1$.
\item For all $0\leq s\leq t$, $N_1^{t-s}\leq 2$.  
\end{enumerate}
Item (i) holds by definition.  Item (ii) follows because every element has at most $2$ successors.  Item (iii) follows directly from (i), (ii) and the fact that for all $s$ and $\ell$, $N^s_{\ell}=x_{\ell}^s+y_{\ell}^s$.  Item (iv) follows from the fact that the only element of height $0$ is $a_{<>}$, which has height $t$.  Item (v) follows from the fact that if for some $\ell$, if $N_{\ell}^t \geq 2$, then we would have $t(a_{\langle \rangle})\geq t+1$.  Item (vi) is because the tree is binary, so the second level can have at most two elements.

We now show that for each $0\leq s\leq t$ and $0\leq \ell <h$, $N_{\ell+1}^{t-s}\leq (2(\ell+1))^{s}$.  If $s=0$ this follows immediately from (v).

Case $s=1$: We want to show for all $0\leq \ell <h$, $N_{\ell+1}^{t-1}\leq (2(\ell+1))^{s}$.  The case where $\ell =0$ is done by (vi).  Let $\ell> 0$ and suppose by induction $N_{\ell}^{t-1}\leq 2 \ell$.  By (iii), (v) and our induction hypothesis,
$$
N_{\ell+1}^{t-1}\leq N_{\ell}^{t-1}+2N_{\ell}^t \leq 2\ell+2=2(\ell+1).
$$

Case $s>1$: Suppose by induction that for all $0\leq s'< s$, the following holds: for all $0\leq \ell<h$, $N_{\ell+1}^{t-s'}\leq (2(\ell+1))^{s'}$. We want to show that for all $0\leq \ell <h$, $N_{\ell+1}^{t-s}\leq (2(\ell+1))^{s}$.  The case $\ell=0$ is done by (vi).  Let $\ell >0$ and suppose by induction that for all $0\leq \ell'<\ell$, $N_{\ell'+1}^{t-s}\leq (2(\ell'+1))^{s}$.  Then by (iii) and our induction hypothesis,
$$
N_{\ell+1}^{t-s}\leq N_{\ell}^{t-s}+2N_{\ell}^{t-s+1} \leq (2\ell)^{s}+2(2\ell)^{s-1}=(2\ell)^s\Big(\frac{\ell+1}{\ell}\Big)\leq (2(\ell+1))^{s}.
$$

Therefore, for all $0\leq \ell <h$,  
$$
N_{\ell+1}\leq \sum_{0\leq s\leq t}N_{\ell+1}^s\leq \sum_{0\leq s\leq t}(2(\ell+1))^{s} \leq t(2(\ell+1))^{t}\leq t(2h)^t.
$$
This implies that
$$
n=N_0+\sum_{0\leq \ell<h}N_{\ell+1} \leq 1+\sum_{0\leq \ell<h}t(2h)^t \leq t(2h)^{t+1}
$$

Rearranging this we obtain that
$$
\frac{(n/t)^{\frac{1}{t+1}}}{2}\leq h.
$$
Since $t\leq t(G)$ this implies $\frac{(n/t(G))^{\frac{1}{t(G)+1}}}{2}\leq h$.  This finishes the proof.
\end{proof}

\noindent Combining Theorem \ref{3.5} and Lemma \ref{indset} immediately implies the following.
\begin{corollary}\label{treecormain}
Suppose $G=(V,E)$ is a graph with tree rank $t$ and $n$ vertices.  Then $G$ contains a complete or independent set of size at least $\frac{(n/t)^{\frac{1}{t+1}}}{4}$.
\end{corollary}


\section{Finitary proof leveraging Theorem \ref{3.5}}\label{finitary}

The following is an adaptation of Proposition 3.1 \cite{CKOS}.
\begin{proposition}\label{new3.1}
Suppose $G=(V,E)$ has tree height $t\geq R(n_1, n,n,n_2)$ witnessed by $T\subseteq V$ and the indexing $T=\{a_{\eta}: \eta \in 2^{<t}\}$ which is a type tree.  Then $G[T]$ contains one of the following as an induced subgraph.
\begin{enumerate}[(i)]
\item a thin spider with $n$ legs,
\item the bipartite half-graph of height $n$,
\item the disjoint union of $n_1$ copies of $K_2$, denoted by $n_1K_2$, or
\item the half split graph of height $n_2$.
\end{enumerate}
\end{proposition}

\begin{proof}
Consider the sets $A=\{a_{<>}, a_0,\ldots a_{0^{t-1}}\}$ and $B=\{a_1, a_{01}, \ldots, a_{0^{t-1}\wedge 1} \}$.  Rename the elements of $A$ and $B$ so that $\langle a_{<>}, a_0,\ldots, a_{0^{t-1}}\rangle = \langle x_1,x_2,\ldots, x_t\rangle$ and $\langle a_1, a_{01}, \ldots, a_{0^{t-1}\wedge 1}\rangle =\langle y_1,y_2,\ldots, y_t\rangle$.  Note that by definition of a standard type tree and our choice of $A$, we have the following.
\begin{itemize}
\item $A$ is an independent set.
\item For each $i\in [t]$, $x_iy_i\in E$.
\item  For each $i<j$, $x_iy_j\notin E$.  
\end{itemize}
We now define a coloring of the edges of the complete graph with vertex set $[t]$ with colors $(a,b)\in \{0,1\}^2$.  Given $i<j\in [t]$, define the color $(a,b)$ of the edge $ij$ as follows.  Set $a=1$ if and only if $x_jy_i\in E$ and $b=1$ if and only if $y_iy_j\in E$.  By Ramsey's theorem, there is a subset $I\subseteq [t]$ such that all the edges of $I$ have the same color $(a,b)$ and the following holds.
\[
|I|=\begin{cases}
n_1 & \text{ if }(a,b)=(0,0)\\
n & \text{ if }(a,b)=(0,1)\\
n & \text{ if }(a,b)=(1,0)\\
n_2 & \text{ if }(a,b)=(1,1)
\end{cases}
\]
Set $Z=\{x_i: i\in I\}\cup \{y_i: i\in I\}$.  Then if $(a,b)=(0,0)$, $G[Z]$ forms an induced copy of $n_1K_2$.  If $(a,b)=(0,1)$, then $G[Z]$ forms an induced copy of a thin spider with $n$ legs.  If $(a,b)=(1,0)$, then $G[Z]$ forms an induced copy of a bipartite half-graph of height $n$.  Finally if $(a,b)=(1,1)$, then $G[Z]$ forms an induced copy of the half split graph of height $n_2$.
\end{proof}

\begin{remark}
\vspace{.1mm}
\begin{enumerate}
\item In the proof of Proposition \ref{new3.1}, we could also have built our configuration over a complete set by instead taking $A=\{a_{<>}, a_1, a_{11}, \ldots, a_{1^{t-1}}\}$ and $B=\{a_0, a_{10}, \ldots, a_{1^{t-1}\wedge 0}\}$.
\item If we don't care whether we build over complete or empty sets, then what Proposition \ref{new3.1} uses is the length of the longest ``straight path" through the tree consisting of nodes with two children, which is at least the tree rank.
\end{enumerate}
\end{remark}

\begin{corollary}\label{treecor}
Suppose $G$ is a prime graph with tree height $t\geq R(h(n,g(n),n), n,n,g(n))$.  Then $G$ contains one of the following or the compliment of one of the following as an induced subgraph. 
\begin{enumerate}[(1)]
\item The $1$-subdivision of $K_{1,n}$ (denoted by $K_{1,n}^{(1)}$). 
\item The line graph of $K_{2,n}$ (denoted by $L(K_{2,n})$).
\item The thin spider with $n$ legs.
\item The bipartite half-graph of height $n$.
\item The graph $H'_{n,I}$.
\item the graph $H_n^*$.
\item A prime graph induced by a chain of length $n$.
\end{enumerate}
\end{corollary}
\begin{proof}
If $G$ contains a chain of length $n+1$, we are done.  So assume this is not the case.  Apply Proposition \ref{new3.1} with $n_1=h(n, g(n), n)$ and $n_2=g(n)$.  In outcomes \ref{new3.1}.(i) and \ref{new3.1}.(ii), we are done.  If $G$ contains a half split graph of height $g(n)$ apply Proposition \ref{5.1} to obtain $H'_{n,I}$ or $H_n^*$.  So assume now $G$ contains no half split graph of height $g(n)$.  The only possible outcome left is \ref{new3.1}.(iii), i.e., that $G$ contains an induced matching with $n_1=h(n,g(n),n)$ edges.  Combining this with our assumptions that $G$ is prime, contains no chains of length $n+1$, and contains no half split graph of height $g(n)$, we have that Proposition \ref{4.1} implies $G$ contains a copy of $K^{(1)}_{1,n}$, the bipartite half-graph of height $n$, $\overline{L(K_{2,n})}$, or a spider with $n$ legs.  This finishes the proof.
\end{proof}

\noindent We now prove Theorem \ref{mainthm} with a value for $N$ which is asymptotically much smaller than $N_{\ref{mainthm}}$.

\begin{theorem}
Let $n\geq 2$ and recall
$$
m=f(n, h(n,g(n), n), g(n)) = 2^{R(n+h(n,g(n),n), 2n-1,n+g(n), n+g(n)-1)+1}.
$$
Suppose 
$$
N= R(h(n,g(n),n), n,n,g(n))(5m)^{R(h(n,g(n),n), n,n,g(n))+1},
$$
and $G$ is a prime graph with at least $N$ vertices.  Then the conclusion of Theorem \ref{mainthm} holds.  Moreover, for large $n$, 
$$
N<<R(m,m)=N_{\ref{mainthm}}.
$$
\end{theorem}
\begin{proof}
Suppose $G$ is a prime graph with at least $N$ vertices.  Suppose first that the tree height, $t=t(G)$ is at least $R(h(n,g(n),n), n,n,g(n))$.  Then Corollary \ref{treecor} implies $G$ contains one of the desired configurations, so we are done.  Assume now that $t\leq R(h(n,g(n),n), n,n,g(n))$.  Remark \ref{reduction} and Proposition \ref{cor2.3} imply that that if $G$ contains a complete or independent set of size $m$ then the conclusion of Theorem \ref{mainthm} holds.  We show $G$ contains a complete or independent set of size $m$.  By Corollary \ref{treecormain}, $G$ contains a complete or independent independent set $I$ such that $|I|\geq \frac{(N/t)^{\frac{1}{t+1}}-2}{4}$, so it suffices to show that $\frac{(N/t)^{\frac{1}{t+1}}}{4}\geq m$.   By definition of $N$ and our assumption on $t$, $N\geq t(5m)^{t+1}$.  This implies $\frac{(N/t)^{\frac{1}{t+1}}}{4}\geq \frac{5m}{4}\geq m$.  This finishes the proof that the conclusion of Theorem \ref{mainthm} holds.  We've now left to show that $N<<N_{\ref{mainthm}}$.  Let $x=R(h(n,g(n),n), n, n, g(n))$.  Then we want to show that large $n$, $x (5m)^{x+1}<<R(m,m)$.  Note that $x\leq \log_2 m$ and recall that by \cite{spencer1}, as long as $m\geq 2$, $R(m,m)\geq (\sqrt{2})^m$.  Combining these facts, we have that the following holds for large $m$ (equivalently, for large $n$).
\begin{align*}
x(5m)^{x+1} \leq (\log_2m) (5m)^{2\log_2 m+1}<<(\sqrt{2})^m\leq R(m,m).
\end{align*}
\end{proof}

\begin{remark}
The theorem uses the fact that any graph $G$ contains a complete or independent set of size $\max\{t(G), h(G)/2\}$,  the inverse relationship between $t(G)$ and $h(G)$ from Theorem \ref{3.5}, and the fact that a binary type tree contains the building blocks of the desired configurations.  These ingredients, i.e. Theorem \ref{new3.1}, Lemma \ref{indset}, and Theorem \ref{3.5}, hold for arbitrary graphs.  
\end{remark}

\section{An infinitary proof}\label{infinitary}
In this section we prove an analogue of Theorem \ref{mainthm} in the infinite setting, and show it implies the finite version, although without the explicit bounds.  Throughout this section we work in the first-order language of graphs, $\mathcal{L}=\{E(x,y)\}$, and employ standard model theoretic notation.  Given sets $A$ and $B$, we will write $AB$ as shorthand for $A\cup B$, and given a tuple of elements $\abar$, we will often write $\abar$ to mean the set of elements in the tuple.  The following proposition is proved in \cite{CKOS} in the setting of finite graphs, but the proof presented there also holds in the setting of infinite graphs.  Given an integer $n$, we will write $R(n)$ to mean $R(n,n)$.
\begin{proposition}[Proposition 2.1 in \cite{CKOS}]\label{chainsandmodules}
Suppose $G$ is a graph and $I\subseteq V(G)$ is a set with at least two vertices, and suppose $v\in V(G)\setminus I$.  Then $G$ has a chain from $I$ to $v$ if and only if all modules containing $I$ as a subset contain $v$.
\end{proposition}
\noindent A useful and straightforward corollary of this is the following.
\begin{corollary}\label{chainsandprimeness}
A graph $G=(V,E)$ is prime if and only if for every set of pairwise distinct vertices $\{x_1,x_2, x_3\}\subseteq V$, there is chain from $\{x_1,x_2\}$ to $x_3$ in $G$.
\end{corollary}
\begin{proof}
Suppose $G=(V,E)$ is a prime graph and $x_1,x_2, x_3\in V$ are pairwise distinct vertices.  Suppose there is no chain from $\{x_1,x_2\}$ to $x_3$.  Then by Proposition \label{chainsandmodules}, there is a module $I$ containing $\{x_1,x_2\}$ as a subset and not containing $v$.  But now $I$ is a nontrivial module, contradicting that $G$ is prime.

Conversely, suppose for every set $\{x_1,x_2, x_3\}\subseteq V$ of pairwise distinct vertices, there is chain from $\{x_1,x_2\}$ to $x_3$ in $G$.  We show that any module $I$ in $G$ is either a singleton or all of $V$.  Suppose by contradiction $I$ is a module which is neither a singleton, nor all of $V$.  Then there are $x_1\neq x_2\in I$ and $x_3\in V\setminus I$.  By assumption there is a chain from $\{x_1,x_2\}$ to $x_3$, so Proposition \ref{chainsandmodules} implies that every module containing $\{x_1,x_2\}$ also contains $x_3$.  In particular, $x_3\in I$, a contradiction.
\end{proof}

\begin{definition} Fix an integer $n\geq 1$.
\begin{enumerate}[(1)]
\item Let $\phi_n(x,y,z)$ be the formula saying that there exists a chain of length at most $n$ from $\{x,y\}$ to $z$.  
\item Let $\psi_n$ be the sentence saying that for any pairwise distinct $x_1,x_2,x_3$, there is a chain of length at most $n$ from $\{x_1,x_2\}$ to $x_3$, i.e. the sentence
$$
\forall x_1x_2x_3 \Bigg(\Bigg(\bigwedge_{1\leq i\neq j\leq 3}x_i\neq x_j \Bigg)\rightarrow \phi_n(x_1,x_2,x_3)\Bigg).
$$
\item Let $\sigma_n$ be the sentence saying that there exists a copy of $H_n$ or a copy of $\overline{H_n}$ as an induced subgraph.
\item Let $\theta_n$ be the sentence saying there exists a copy of $H_{n,I}'$, $H_n^*$ or $\overline{H_n^*}$.
\item Let $\rho_n$ be the sentence which says that one of the following or the compliment of one of the following appears as induced subgraph: $K^{(1)}_{1,n}$, $L(K_{2,n})$, a spider with $n$ legs.
\end{enumerate}
\end{definition}

Given $k\geq 1$, we will call a graph $G$ $k$-edge-stable if $G$ omits all half-graphs of height $k$.  We will call $G$ edge-stable when it is $k$-edge stable for some $k$ (equivalently, when its edge relation is a stable formula).  Call a subset of $I$ of $G$ \emph{edge indiscernible} if it is indiscernible with respect to the edge relation.  We remark that Proposition \ref{5.1} applies in the case of an infinite prime graph as well as a finite one, via exactly the same proof as in \cite{CKOS}.  Given a formula $\phi$, we let $\phi^1=\phi$ and $\phi^0=\neg \phi$.  We now recall a definition and claim from \cite{MS}.

\begin{definition}
Given $\ell\geq 2$, let $\Delta_{\ell}=\{E(x_0,x_1)\}\cup \{\phi^i_{\ell, m}: m\leq \ell, i\in \{0,1\}\}$,
where
$$
\phi^i_{\ell, m}=\phi_{\ell, m}^i(x_0,\ldots, x_{\ell-1})=\exists y\Bigg(\bigwedge_{j<\ell}E(x_j,y)^{\text{ if }i=0} \wedge \bigwedge_{m\leq j\leq \ell}E(x_j,y)^{\text{ if }i=1}\Bigg).
$$
\end{definition}

\begin{claim}[Claim 3.2 of \cite{MS}]\label{MS3.2}
Suppose $G$ is an $\ell$-edge stable graph.  Suppose $m\geq 4\ell$ and $\langle a_i: i<\alpha \rangle$ is a $\Delta_{\ell}$-indiscernible sequence in $G$, and $b\in G$.  Then either $|\{i: E(a_i, b)\}|<2\ell$ or $|\{i: \neg E(a_i, b)\}|<2\ell$.
\end{claim}

\begin{proposition}\label{infinite3.1}
For any integer $n\geq 1$, any infinite graph satisfying $\psi_n\wedge \neg \sigma_n\wedge \neg \theta_{n}$ is prime, edge-stable, and contains one of the following or the compliment of one of the following as an induced subgraph.
\begin{enumerate}[(1)]
\item A spider with $\omega$ many legs,
\item $L(K_{2,\omega})$,
\item A perfect matching of length $\omega$.
\end{enumerate}
\end{proposition}
\begin{proof}
Since $G\models \psi_n$, Corollary \ref{chainsandprimeness} implies $G$ is prime.  Set $\ell=R(R(g(n)))$. We show $G$ is $\ell$-edge-stable.  Suppose by contradiction $G$ contains a half-graph $a_1b_1,\ldots, a_{\ell}b_{\ell}$ so that $E(a_i,b_j)$ if and only if $i\leq j$.  By Ramsey's theorem, there is a complete or independent set $A\subseteq \{a_1,\ldots, a_{\ell}\}$ such that $|A|=R(g(n))$.  By reindexing, assume $A=\{a_1,\ldots, a_{R(g(n))}\}$.  Applying Ramsey's theorem again, we have that there is a complete or independent set $B'\subseteq\{b_1,\ldots, b_{R(g(n))}\}$ such that $|B'|=g(n)$.  By reindexing, assume $B'=\{b_1,\ldots, b_{g(n)}\}$.  Then $a_1b_1,\ldots,a_{g(n)}b_{g(n)}$ forms an induced copy of $H_{g(n)}$, $\overline{H_{g(n)}}$, or a half split graph of height $g(n)$.  Since $G\models \neg \sigma_n$, it must contain a half split graph of height $g(n)$.  By Proposition \ref{5.1}, $G$ contains an induced copy of $H_{n,I}'$, $H_n^*$, or $\overline{H_n^*}$, contradicting that $G\models \neg \theta_n$.  Therefore $G$ is $\ell$-edge-stable.

By Ramsey's theorem there is an infinite $\Delta_{\ell}$-indiscernible sequence $I=\{c_i: i<\omega\}$ in $G$.  Note $I$ is a complete or independent set.  Without loss of generality, assume it is independent (otherwise we obtain the compliments everything that follows).  Claim \ref{MS3.2} implies that for all $b\notin I$, either $|\{c_i: E(b,c_i)\}|\leq 2\ell$ or $|\{c_i: \neg E(b,c_i)\}|\leq 2\ell$.  Given $b\notin I$, set 
\[
f(b)=\begin{cases}
1&\text{ if }|\{c_i:  E(b,c_i)\}|\leq 2\ell \\
0&\text{ if }|\{c_i: \neg E(b,c_i)\}|\leq 2\ell,
\end{cases}
\]
and set $S_b=\{c_i: E(b,c_i)^{f(b)}\}$.  We construct two sequences $J_1=\{a_i: i<\omega\}$ and $J_2=\{b_i: i<\omega\}$ along with a sequence of sets $\{A_i: i<\omega\}$ with the following properties.
\begin{enumerate}[$\bullet$]
\item For each $k<\omega$, $b_k\notin Ib_1\ldots b_{k-1}$ and $a_k\in S_{b_k}$,
\item for each $i,j< \omega$, $E(b_i,a_j)^{f(b_i)} \Leftrightarrow i= j$,
\item $I\supseteq A_1\supseteq A_2 \supseteq \ldots$ and for each $k<\omega$, $|A_k|=\omega$,
\item For each $j\leq k<\omega$, $A_k\cap S_{b_j}=\emptyset$.
\end{enumerate}

Step $0$: Since $I$ is not a module, there is a vertex $b_1$ which is mixed on $I$.  Note that since $I$ is edge-indiscernible, we must have that $b_1\notin I$.  Choose $a_1\in S_b$ and set $A_1=I\setminus S_{b_1}$.  Note that since $|I|=\omega$ and $a_1S_{b_1}$ is finite, $|A_1|=\omega$.

Step $k$: Suppose now we've constructed $b_1a_1,\ldots, b_{k-1}a_{k-1}$, and $A_1,\ldots, A_{k-1}$ satisfying the desired hypotheses.  Since $A_{k-1}$ is not a module, there is $b_k$ which is mixed on $A_{k-1}$.  In other words, $A_{k-1}\cap S_{b_k}\neq \emptyset$.  Since $I$ is edge-indiscernible, $b_k$ is not in $I$.  For each $j<k$, $A_{k-1}\cap S_{b_j}=\emptyset$ implies $b_j$ is not mixed on $A_{k-1}$.  Therefore $b_k \notin \{b_1\ldots b_{k-1}\}$.  Choose $a_k\in S_{b_k}\cap A_{k-1}$ and set $A_k= A_{k-1}\setminus a_kS_{b_k}$.  Note that by our induction hypothesis, $|A_{k-1}|=\omega$ and by definition $a_kS_{b_k}$ is finite, so $|A_k|=\omega$.  This completes the construction.

By Ramsey's theorem, there are infinite subsequences $I_1=(a'_i)_{i<\omega} \subseteq (a_i)_{i<\omega}$ and $I_2=(b'_i)_{i<\omega} \subseteq (b_i)_{i<\omega}$ such that $I_1I_2=(a'_ib'_i)_{i<\omega}$ is edge-indiscernible.  If $I_2$ is a complete set and $f(b_1')=0$, then $I_2I_2$ is a thick spider with $\omega$ many legs.  If $I_2$ is a complete set and $f(b_1')=1$, then $I_1I_2$ is a thin spider with $\omega$ many legs.  If $I_2$ is an independent set and $f(b'_1)=0$, then $I_1I_2$ forms a copy of $\overline{L(K_{2,\omega})}$.  Therefore we are left with the case when $I_2$ is an independent set and $f(b_1')=1$.  In this case $I_1I_2$ forms a perfect matching of length $\omega$.
\end{proof}

\noindent The following argument is an infinitary version of the argument used to prove Proposition \ref{4.1} in \cite{CKOS}.
\begin{proposition}\label{infinite4.1}
Suppose $G$ is an infinite, prime, edge-stable graph satisfying $\psi_n$ and suppose $M$ is an infinite perfect matching in $G$.  Then $G$ contains of one of the following or the compliment of one of the following as an induced subgraph.
\begin{enumerate}[(1)]
\item $K^{(1)}_{1,\omega}$,
\item $L(K_{2,\omega})$,
\item A spider with $\omega$-many legs.
\end{enumerate}
\end{proposition}
\begin{proof}
Suppose $G$ is an infinite, prime, edge-stable graph satisfying $\psi_n$ and suppose $M$ is an infinite perfect matching $M$ in $G$.  Since $M$ is not prime, $V(G)\setminus V(M)\neq \emptyset$.  Since $G$ is prime and satisfies $\psi_n$, Corollary \ref{chainsandprimeness} implies that for every $v\in V(G)\setminus V(M)$ there is an integer $t(v)\leq n$ such that there is a chain of length less than or equal to $t(v)$ from $v$ to $e$ for infinitely many $e\in M$.  Set $t=t(M)=\min \{t(v): v\in V(G)\setminus V(M)\}$.  We show by induction on $2\leq t\leq n$ that the conclusion of the proposition is true.

Fix $v\in V$ such that $t(v)=t$ and an infinite $M'\subseteq M$ such that there is a chain of length at most $t$ from $v$ to $e$ for every $e\in M'$.  Suppose first that $t=2$.  Then $vM'$ is isomorphic to $K^{(1)}_{1,\omega}$ and we are done.  Assume now $2<t\leq n$ and suppose by induction that for all $2\leq t'<t$, if $G$ contains an infinite perfect matching $M''$ with $t(M'')=t'$, then the conclusion of the proposition holds.  Enumerate $M'=\{x_iy_i: i<\omega\}$ and delete the edges $e\in M'$ on which $v$ is mixed. Since $t>2$, we have deleted only finitely many elements of $M'$.  For each $i<\omega$ choose a chain $C_{x_iy_i}=\{v_0,v_1,\ldots, v_n\}$ from $x_iy_i$ to $v$ (so $\{x_i,y_i\}=\{v_0,v_1\}$).  Set set $z_i=v_2$.  

Note by assumption, $v$ is not mixed on any $x_iy_i$, so $z_i\neq v$, and since $M'$ is a matching, $z_i\notin M'$.  By Ramsey's theorem, the sequence $(x_iy_iz_i)_{i<\omega}$ contains an infinite indiscernible sequence $(x_i'y_i'z_i')_{i<\omega}$.  Since $t>2$, we must have that for each $i<\omega$, $z_i'$ is not mixed on $x_j'y_j'$ for all $j\neq i$, so in particular, $E(z_1', x_2')\equiv E(z_1', y_2')$.  Since $G$ is edge-stable, we have that $E(z_1'x_2') \equiv E(z_2'x_1')$ and $E(z_1'y_2') \equiv E(z_2'y_1')$.  Combining all of this, we have 
$$
E(z_2'x_1')\equiv E(z_1'x_2')\equiv E(z_1'y_2') \equiv E(z_2'y_1').
$$
By relabeling if necessary, we may assume $E(z_1'y_1')$ and $\neg E(z_1',x_1')$.  By indiscernibility and our assumptions, the type of $(x_i'y_i'z_i')_{i<\omega}$ depends only on $E(z'_1,x_2')$ and $E(z'_1,z_2')$.  Suppose first that $E(z_1',z_2')$, so $(z_i')_{i<\omega}$ is a complete set.  If $E(z_1',x_2')$, then $(z'_i,x_i')_{i<\omega}$ is a thick spider with $\omega$ many legs.  If $\neg E(z_1',x_2')$, then $(z'_i,y'_i)_{i<\omega}$ is a thin spider with $\omega$ many legs.  

Suppose now that $\neg E(z_1',z_2')$, so $(z_i')_{i<\omega}$ is an independent set.  If $E(z_1',x_2')$, then $(z_i', x_i')_{i<\omega}$ is a copy of $\overline{L(K_{2,\omega})}$.  If $\neg E(z_1',x_2')$, then $M'':=(z_i',y_i')_{i<\omega}$ is an infinite perfect matching.  In this case, we now have that for each $i<\omega$, $C_{x'_iy_i'}\setminus \{x_i'\}$ is a chain of length at most $t-1$ from $\{z_i',y_i'\}$ to $v$, that is $t(M'')=t-1$.  By our induction hypothesis, $G$ satisfies the conclusion of the proposition.
\end{proof}

\noindent We now prove a version of Theorem \ref{mainthm} for infinite graphs, then use it to prove Theorem \ref{mainthm}.
\begin{theorem}
An infinite prime graph $G$ contains one of the following.
\begin{enumerate}
\item Copies of $H_n$, $\overline{H_n}$, $H_n^*$, $\overline{H_n^*}$, $H'_{n,I}$, or $\overline{H'_{n,I}}$ for arbitrarily large finite $n$,
\item Prime graphs induced by arbitrarily long finite chains,
\item $K^{(1)}_{1,\omega}$ or its compliment,
\item $L(K_{2,\omega})$ or its compliment,
\item A spider with $\omega$ many legs.
\end{enumerate}
\end{theorem}
\begin{proof}
Suppose $G$ is an infinite prime graph which fails 1 and 2.  Since $G$ is prime but fails 2, Proposition \ref{cor2.3} implies $G$ does not contain arbitrarily long finite chains.  Thus there is $n_1\in \mathbb{N}$ such that $G\models \psi_{n_1}$.  Since $G$ fails 1, there is $n_2$ such that $G$ contains no copy of $H_{n_2}$, $H_{n_2}^*$, $\overline{H_{n_2}^*}$, or $H'_{n_2,I}$. Let $n_3=\max\{n_1,n_2\}$, then $G$ is prime and satisfies $\phi_{n_3}\wedge \neg \sigma_{n_3} \wedge \neg \theta_{n_3}$.  Applying Corollary \ref{infinite3.1}, we have that either $G$ satisfies 5 or 4, or $G$ contains an induced perfect matching of length $\omega$.  If $G$ contains an induced perfect matching of length $\omega$, Proposition \ref{infinite4.1} implies $G$ satisfies 3, 4, or 5.
\end{proof}

{\bf Proof of Theorem \ref{mainthm}}  Fix $n\geq 1$.  By definition, any finite prime graph $G$ satisfying $\sigma_n$ or $\theta_{n}$ contains one of the desired configurations.  If a finite prime graph $G$ of size at least $3$ satisfies $\neg \psi_n$, then $G$ contains three distinct points $x,y,z$ such that there is no chain of length less than or equal to $n$ from $\{x,y\}$ to $z$.  Corollary \ref{chainsandprimeness} implies that there is some chain from $\{x,y\}$ to $z$.  Therefore there is a chain $v_0,\ldots, v_t$ of length $t\geq n+1$ from $\{x,y\}$ to $z$.  Since initial sequences of chains are chains, $v_0,\ldots, v_{n+1}$ is a chain of length $n+1$.  By Proposition \ref{cor2.3}, $G$ contains a chain of length $n$ inducing a prime subgraph.  So if $G$ has size at least $3$ and satisfies $\sigma_n\vee \theta_n\vee \neg \psi_n$, we are done.

We now show there is $N$ such that any finite prime graph of size at least $N$ satisfying $\neg \sigma_n\wedge \neg \theta_{g(n)}\wedge \psi_n$ must also satisfy $\rho_n$.  This combined with the above finishes the proof.  Suppose by contradiction that no such $N$ exists.  Then there are arbitrarily large finite graphs which satisfy $\neg \sigma_n\wedge \neg \theta_{n}\wedge \psi_n\wedge \neg \rho_n$, so by compactness there is an infinite graph $G$ satisfying $\neg \sigma_n\wedge \neg \theta_{n}\wedge \psi_n\wedge \neg \rho_n$.  By Proposition \ref{infinite3.1}, $G$ is edge-stable and contains an infinite perfect matching.  But then Proposition \ref{infinite4.1} clearly implies $G\models \rho_n$, a contradiction.
\qed

\bibliography{/Users/carolineterry/Desktop/papers/science1.bib}
\bibliographystyle{amsplain}

\end{document}